\documentclass[11pt]{amsart}

\usepackage{amssymb}
\usepackage[abbrev]{amsrefs}

\newcommand{\N}{\mathbb N}
\newcommand{\R}{\mathbb R}
\providecommand{\abs}[1]{\left\lvert#1\right\rvert}
\providecommand{\norm}[1]{\left\lVert#1\right\rVert}
\providecommand{\bignorm}[1]{\bigl\lVert#1\bigr\rVert}
\providecommand{\ip}[2]{\left\langle #1, #2 \right\rangle}
\newcommand{\A}{\operatorname{A}}
\newcommand{\F}{\operatorname{F}}
\newcommand{\Fhat}{\hat{\mathrm F}}

\allowdisplaybreaks
\numberwithin{equation}{section}

\theoremstyle{plain}
\newtheorem{theorem}{Theorem}[section]
\newtheorem{lemma}[theorem]{Lemma}
\newtheorem{example}[theorem]{Example}
\newtheorem{corollary}[theorem]{Corollary}

\theoremstyle{remark}
\newtheorem{remark}[theorem]{Remark}

\title[quasinonexpansive extension]{%
A quasinonexpansive extension of a mapping with an attractive point in a
Hilbert space}

\author{Koji~Aoyama}
\address[K.~Aoyama]
{Aoyama Mathematical Laboratory,
Konakadai, Inage-ku, Chiba, Chiba 263-0043, Japan}
\email{aoyama@bm.skr.jp}

\keywords{Attractive point, 
quasinonexpansive extension, 
generalized hybrid mapping, approximation algorithm}
\subjclass[2010]{47J25, 47J20, 47H09}

\begin{document}

\begin{abstract}
 In this paper, we show that, under appropriate conditions, there
 exists a quasinonexpansive extension of a mapping with an attractive
 point in the sense of Takahashi and Takeuchi~\cite{MR2858695} such
 that the fixed point set of the extension equals the attractive point
 set of the given mapping. 
 Then using the quasinonexpansive extension, 
 we establish some convergence theorems for approximating attractive
 points of a generalized hybrid mapping in the sense of Kocourek,
 Takahashi, and Yao~\cite{MR2761610}.
\end{abstract}

\maketitle

\section{Introduction}

Let $H$ be a Hilbert space, $C$ a subset of $H$, and $T\colon C\to H$ a
mapping.
Takahashi and Takeuchi~\cite{MR2858695} introduced the notion of an
attractive point of $T$; see \S2 for the definition of an attractive
point. 
It is easy to verify that if $T$ is quasinonexpansive, then 
every fixed point of $T$ is an attractive point of $T$. 
Thus an attractive point is regarded as a generalization of a fixed
point for a quasinonexpansive mapping. 

Takahashi and Takeuchi~\cite{MR2858695} also established a mean
convergence theorem for an attractive point of a generalized hybrid
mapping in the sense of Kocourek et al.~\cite{MR2761610}; 
see \S2 for the definition of a generalized hybrid mapping. 
Such a mapping originates from a $\lambda$-hybrid mapping
introduced in Aoyama et al.~\cite{MR2682871}; 
see also \cites{MR2981792,MR3017202}. 
We know some existence and convergence results for attractive points 
of a generalized hybrid mapping and its variants; see, for example, 
\cites{MR2983907, MR3015118, MR3073499, MR3397117}. 

In this paper, 
we prove that, under appropriate conditions, 
if a mapping $T\colon C \to H$ has an attractive point, then there
exists a quasinonexpansive extension $\tilde{T}\colon H\to H$ of $T$
such that the set of fixed points (or asymptotic fixed points) of 
$\tilde{T}$ equals that of attractive points of $T$. 
Then using the quasinonexpansive extension, 
we derive convergence theorems for attractive points from those for 
fixed points of quasinonexpansive mappings. 
Moreover, we also obtain convergence results for attractive points of a
generalized hybrid mapping.

\section{Preliminaries}

Throughout the present paper, 
$H$ denotes a real Hilbert space, 
$\ip{\,\cdot\,}{\,\cdot\,}$ the inner product of $H$, 
$\norm{\,\cdot\,}$ the norm of $H$, 
$C$ a nonempty subset of $H$, 
$I$ the identity mapping on $H$, 
and $\N$ the set of positive integers. 
Strong convergence of a sequence $\{x_n\}$ in $H$ to $z\in H$ is denoted
by $x_n \to z$ and weak convergence by $x_n \rightharpoonup z$. 

Let $T\colon C\to H$ be a mapping. 
Then the set of fixed points of $T$ is denoted by $\F(T)$, that is, 
$\F(T) = \{z \in C\colon Tz=z\}$. 
A point $z \in H$ is said to be an \emph{asymptotic fixed point} of
$T$~\cite{MR1386686} if there exists a sequence $\{x_n\}$ in $C$ such
that $x_n - T x_n\to 0$ and $x_n \rightharpoonup z$. 
The set of asymptotic fixed points of $T$ is denoted by $\Fhat(T)$. 
It is clear that $\F(T) \subset \Fhat(T)$. 
A point $z \in H$ is said to be an \emph{attractive point} of
$T$~\cite{MR2858695} if $\norm{Tx - z} \leq \norm{x-z}$ for all $x \in
C$. 
The set of attractive points of $T$ is denoted by $\A(T)$, that is, 
\[
\A(T) = \bigcap_{x\in C} \{ z \in H\colon \norm{Tx - z} \leq
\norm{x-z}\}. 
\]
It is clear that $C \cap \A(T) \subset \F(T)$, and that 
$\A(T)$ is closed and convex. 

Let $T\colon C\to H$ be a mapping and $F$ a nonempty subset of $H$. 
Then $T$ is said to be \emph{quasinonexpansive with respect to
$F$}~\cite{sqnIII} if $\norm{Tx - z} \leq \norm{x-z}$ for all $x\in C$
and $z\in F$; 
$T$ is said to be \emph{quasinonexpansive} if $\F(T) \ne \emptyset$
and $\norm{Tx - z} \leq \norm{x-z}$ for all $x\in C$ and $z\in \F(T)$; 
$T$ is said to be \emph{nonexpansive} if 
$\norm{Tx - Ty} \leq \norm{x-y}$ for all $x,y\in C$; 
$T$ is said to be \emph{generalized hybrid}~\cite{MR2761610} 
if there exist $\alpha,\beta \in \R$ such that 
\[
 \alpha \norm{Tx-Ty}^2 + (1-\alpha)\norm{x-Ty}^2 \leq 
 \beta \norm{Tx-y}^2 + (1-\beta) \norm{x-y}^2
\]
for all $x,y \in C$; 
$T$ is said to be \emph{demiclosed at $0$} if 
$Tz = 0$ whenever $\{ x_n \}$ is a sequence in $C$ such that 
$x_n\rightharpoonup z$ and $T x_n \to 0$; see, for example, 
\cite{MR1074005}. 
It is clear that 
\begin{itemize}
 \item if $\A(T) \ne \emptyset$, 
       then $T$ is quasinonexpansive with respect to~$\A(T)$;
 \item if $T$ is a generalized hybrid mapping, then 
       $\F(T) \subset \A(T)$;
 \item $I-T$ is demiclosed at $0$ if and only if 
       $\Fhat(T) = \F(T)$. 
\end{itemize}
Moreover, under the assumption that $C$ is closed and convex, 
we know the following:
\begin{itemize}
 \item If $T$ is quasinonexpansive, then $\F(T)$
       is closed and convex; see~\cite{MR0298499}*{Theorem~1};
 \item if $T$ is nonexpansive, then $I-T$ is
       demiclosed at $0$; see~\cite{MR1074005}. 
\end{itemize}

A generalized hybrid mapping has the following property: 

\begin{lemma}[\cite{MR3397117}*{Lemma 3.1}] \label{l:twy15:demiclosed}
 Let $H$ be a Hilbert space, $C$ a nonempty subset of $H$,
 $T\colon C \to H$ a generalized hybrid mapping, 
 and $\{x_n\}$ a sequence in $C$ such that 
 $x_n - Tx_n \to 0$ and $x_n \rightharpoonup z$. 
 Then $z \in \A(T)$, that is, $\Fhat(T) \subset \A(T)$. 
\end{lemma}

Let $D$ be a nonempty closed convex subset of $H$. 
It is known that, for each $x \in H$, 
there exists a unique point $x_0 \in D$ such that 
\[
 \norm{x - x_0} = \min\{\norm{x-y}: y\in D\}. 
\]
Such a point $x_0$ is denoted by $P_D (x)$ and $P_D$ is called the
\emph{metric projection} of $H$ onto $D$.
It is known that the metric projection is nonexpansive; 
see~\cite{MR2548424} for more details. 

The following theorem is a direct consequence
of~\cite{MR2960628}*{Theorem 5.5}; see also~\cite{MR2773244}*{Theorem
3.4}. 

\begin{theorem}\label{t:AKT2012JNAO}
 Let $H$ be a Hilbert space, 
 $T \colon H \to H$ a quasinonexpansive mapping, 
 $\{\alpha_n \}$ a sequence in $(0,1]$, 
 $\{\beta_n \}$ a sequence in $[0,1]$, and 
 $\{x_n\}$ a sequence defined by $u, x_1 \in H$ and 
 \[
 x_{n+1} = \alpha_n u + (1 - \alpha_n) 
  [\beta_n x_n + (1-\beta_n)T x_n]
 \]
 for $n\in \N$. 
 Suppose that $\Fhat(T) = \F(T)$, 
 $\alpha_n \to 0$, $\sum_n \alpha_n = \infty$, and 
 $\liminf_{n} \beta_n (1-\beta_n)> 0$. 
 Then $\{ x_n \}$ converges strongly to $P_{\F(T)}(u)$. 
\end{theorem}

\begin{remark}
 In Theorem~\ref{t:AKT2012JNAO}, 
 the condition $\liminf_{n} \beta_n (1-\beta_n)> 0$ is equivalent to
 the following: $\liminf_{n} \beta_n > 0$ and $\limsup_{n}\beta_n <1$. 
\end{remark}

The following theorem is a direct consequence
of~\cite{MR2058234}*{Theorem 3.2}; 
see also~\cite{MR2529497}. 

\begin{theorem} \label{t:MT2004}
 Let $H$ be a Hilbert space, 
 $T \colon H \to H$ a quasinonexpansive mapping, 
 $\{\alpha_n \}$ a sequence in $[0,1]$, and
 $\{x_n\}$ a sequence defined by $x_1 \in H$ and 
 \[
 x_{n+1} = \alpha_n x_n + (1 - \alpha_n) Tx_n
 \]
 for $n\in \N$. 
 Suppose that $\Fhat(T) = \F(T)$
 and $\liminf_n \alpha_n (1-\alpha_n) > 0$. 
 Then $\{ x_n \}$ converges weakly to some point $w \in \F(T)$. 
\end{theorem}

\section{Quasinonexpansive extensions}

In this section, we prove that, under appropriate assumptions, a mapping
with an attractive point has a quasinonexpansive extension such that 
the set of fixed points (or asymptotic fixed points) is equal to that of
attractive points (Lemma~\ref{l:demiclosed}). 
We begin with the following: 

\begin{lemma}\label{l:extension}
 Let $H$ be a Hilbert space, $C$ a nonempty subset of $H$,
 $T\colon C \to H$ a mapping with an attractive point, 
 and $\tilde{T} \colon H \to H$ a mapping defined by 
 \begin{equation}\label{e:extension}
  \tilde{T}x = 
   \begin{cases}
    Tx, & x\in C;\\
    P_{\A(T)}(x), & \text{otherwise.} 
   \end{cases}
 \end{equation}
 Then $\tilde{T}$ is an extension of $T$ and quasinonexpansive with
 respect to $\A(T)$. 
 Moreover, $\A(T) \subset \F(\tilde{T})$. 
\end{lemma}

\begin{proof}
 By the definition of $\tilde{T}$,  
 it is clear that $\tilde{T}$ is an extension of $T$. 
 We show that $\tilde{T}$ is quasinonexpansive with respect to $\A(T)$. 
 Let $x \in H$ and $z \in \A(T)$. 
 Suppose that $z \in C$. Since $z$ is an attractive point of $T$, we
 have $\bignorm{\tilde{T}x-z}=\norm{Tx-z} \leq \norm{x-z}$. 
 On the other hand, suppose that $z \notin C$. 
 Since $P_{\A(T)}$ is nonexpansive and $z = P_{\A(T)}(z)$, we have
 \[
 \bignorm{\tilde{T} x - z}
 = \norm{P_{\A(T)} (x) - P_{\A(T)} (z)} \leq \norm{x -z}. 
 \]
 Therefore, $\tilde{T}$ is quasinonexpansive with respect to $\A(T)$. 

 We next show that $\A(T) \subset \F(\tilde{T})$. 
 Let $z \in \A(T)$. Suppose that $z \in C$. 
 Since $\tilde{T}$ is quasinonexpansive with respect to $\A(T)$, 
 we have $\bignorm{\tilde{T}z-z} \leq \norm{z-z} = 0$, and hence 
 $z \in \F(\tilde{T})$. 
 On the other hand, suppose that $z \notin C$. 
 Then we have $\tilde{T}z = P_{\A(T)}(z) = z$, and hence $z \in
 \F(\tilde{T})$. 
 As a result, we conclude that $\A(T) \subset \F(\tilde{T})$. 
\end{proof}

\begin{remark}
 In Lemma~\ref{l:extension}, 
 one can verify that $\A(\tilde{T}) = \A(T)$. 
\end{remark}

The following example shows that
$\A(T) \ne \F(\tilde{T})$ in Lemma~\ref{l:extension}. 

\begin{example}
 Let $H=\R$ and $C= \R \setminus \{0\}$. 
 Let $T\colon C \to C$ be a mapping defined by 
 \[
 Tx = 
 \begin{cases}
  1, & x = 1;\\
  -x, &\text{otherwise}. 
 \end{cases}
 \]
 Then $\F(T) = \{ 1\}$ and $\A(T) = \{0\}$. 
 Moreover, let $\tilde{T} \colon H \to H$ be a mapping defined 
 by~\eqref{e:extension}, that is, 
 \[
 \tilde{T}x = 
 \begin{cases}
  0,  & x=0; \\
  Tx, & \text{otherwise}. 
 \end{cases}
 \]
 Then $\F(\tilde{T}) = \{0,1\}$. Therefore, $\A(T) \ne \F(\tilde{T})$. 
\end{example}

\begin{proof}
 The equality $\F(T) = \{1\}$ is obvious. 
 We first show that $\A(T) = \{0\}$. 
 Let $x \in C$. 
 If $x =1$, then 
 $\abs{Tx - 0} = \abs{1} = \abs{x - 0}$; 
 otherwise $\abs{Tx - 0} = \abs{-x} = \abs{x - 0}$. 
 Thus $0 \in \A(T)$. 
 On the other hand, suppose that $z \in \A(T)$ and $z \ne 0$. 
 Then $z \in C$. As a result, we have 
 $z \in C \cap \A(T) \subset \F(T)$, and hence $z=1$. 
 However, since
 \[
  \abs{T(1/2) -1} = \abs{-1/2 -1} = 3/2 > 1/2 = \abs{1/2-1}, 
 \]
 we have $z \notin \A(T)$, which is a contradiction. 
 Therefore we conclude that $\A(T) = \{0\}$. 
 
 We next show that $\F(\tilde{T}) = \{0,1\}$. 
 By definition, $\tilde{T} 0 = 0$ and $\tilde{T}1 = T1 = 1$. 
 Thus $\{ 0,1 \} \subset \F(\tilde{T})$. 
 If $z \notin \{ 0,1\}$, then we have $\tilde{T}z = Tz = -z \ne z$. 
 This means that $\{ 0, 1\} \supset \F(\tilde{T})$. 
\end{proof}

\begin{lemma}\label{l:demiclosed}
 Let $H$ be a Hilbert space, $C$ a nonempty subset of $H$, 
 $T\colon C \to H$ a mapping with an attractive point, 
 and $\tilde{T} \colon H \to H$ a mapping defined by \eqref{e:extension}. 
 Then the following hold: 
\begin{enumerate}
 \item If $\F(T) \subset \A(T)$, then 
       $\A(T) = \F(\tilde{T})$ and $\tilde{T}$ is quasinonexpansive; 
 \item if $\hat{\F}(T) \subset \A(T)$, then 
       $\hat{\F}(\tilde{T}) = \F(\tilde{T})$, that is, 
       $I- \tilde{T}$ is demiclosed at $0$.
\end{enumerate}
\end{lemma}

\begin{proof}
 We first show (1). 
 We know from Lemma~\ref{l:extension} that $\tilde{T}$ is
 quasinonexpansive with respect to $\A(T)$, and that 
 $\A(T) \subset \F(\tilde{T})$. 
 Thus it is enough to show that $\A(T) \supset \F(\tilde{T})$. 
 Let $z \in \F(\tilde{T})$. 
 If $z \in C$, then $z = \tilde{T}z = Tz$, 
 and hence $z \in \F(T) \subset \A(T)$. 
 If $z \notin C$, then $z = \tilde{T}z = P_{\A(T)}(z) \in \A(T)$. 
 Consequently, it turns out that $\A(T) \supset \F(\tilde{T})$. 

 We next show (2). 
 Since $\F(T) \subset \Fhat(T) \subset \A(T)$, 
 it follows from~(1) that $\A(T) = \F(\tilde{T})$. Thus it is enough to
 prove that $\Fhat (\tilde{T}) \subset \A(T)$.
 Let $z \in \Fhat(\tilde{T})$. Then there exists a sequence $\{x_n\}$ in
 $H$ such that $x_n - \tilde{T} x_n \to 0$ and $x_n \rightharpoonup z$. 
 We consider two cases, which might not be exclusive. 
 (i)~Suppose that there exists a subsequence $\{x_{n_i}\}$ of $\{x_n\}$
 such that $x_{n_i} \in C$ for all $i \in \N$. 
 Then it follows that 
 $x_{n_i} - T x_{n_i} = x_{n_i} - \tilde{T} x_{n_i} \to 0$ 
 and $x_{n_i} \rightharpoonup z$. 
 Thus, by assumption, we deduce that $z \in \Fhat(T) \subset \A(T)$. 
 (ii)~Suppose that there exists a subsequence
 $\{x_{n_i}\}$ of
 $\{x_n\}$ such that $x_{n_i} \notin C$ for all $i \in \N$. 
 Then $x_{n_i} - P_{\A(T)}(x_{n_i}) 
 = x_{n_i} - \tilde{T} x_{n_i} \to 0$ and $x_{n_i} \rightharpoonup z$. 
 Since $P_{\A(T)}$ is a nonexpansive mapping on $H$, 
 $I-P_{\A(T)}$ is demiclosed at $0$. 
 Hence $z \in \F(P_{\A(T)}) = \A(T)$.
 This completes the proof. 
\end{proof}

\section{Approximation of attractive points}

In this section, 
using lemmas in the previous section (Lemmas~\ref{l:extension}
and~\ref{l:demiclosed}) and convergence theorems for quasinonexpansive
mappings (Theorems~\ref{t:AKT2012JNAO} and~\ref{t:MT2004}), 
we obtain two convergence theorems for attractive points of 
a mapping satisfying the condition that every asymptotic fixed point is
an attractive point, and as corollaries of them, we also obtain
convergence results for attractive points of generalized hybrid
mappings. 

\begin{theorem}\label{t:strong}
 Let $H$ be a Hilbert space, $C$ a nonempty convex subset of $H$, 
 $T\colon C \to C$ a mapping with an attractive point, 
 $\{\alpha_n \}$ a sequence in $(0,1]$, 
 $\{\beta_n\}$ a sequence in $[0,1]$, 
 and $\{x_n\}$ a sequence in $C$ defined by $u, x_1 \in C$ and
 \begin{equation}\label{e:halpern}
  x_{n+1} = \alpha_n u + (1-\alpha_n) [\beta_n x_n + (1-\beta_n)Tx_n]   
 \end{equation}
 for $n \in \N$. 
 Suppose that $\sum_n \alpha_n = \infty$, 
 $\lim_n \alpha_n = 0$, and $\liminf_n \beta_n (1-\beta_n) >0$. 
 If $\Fhat(T) \subset \A(T)$, 
 then $\{x_n\}$ converges strongly to $P_{\A(T)}(u)$. 
\end{theorem}

\begin{proof}
 Let $\tilde{T}\colon H\to H$ be an extension of $T$
 defined by~\eqref{e:extension}. 
 By the assumption that $\Fhat(T) \subset \A(T)$, 
 we see that $\F(T) \subset \Fhat(T) \subset \A(T)$. 
 Thus Lemma~\ref{l:demiclosed} implies that
 $\Fhat(\tilde{T}) = \F(\tilde{T}) = \A(T)$ and 
 $\tilde{T}$ is quasinonexpansive. 
 Moreover, since $C$ is convex and $\tilde{T}$ is an extension of $T$, 
 it follows that
 \[
 x_{n+1} = \alpha_n u + (1-\alpha_n) [\beta_n x_n + (1-\beta_n)
 \tilde{T}x_n]   
 \]
 for all $n \in \N$. 
 Therefore we deduce from Theorem~\ref{t:AKT2012JNAO} that 
 $x_n \to P_{\F(\tilde{T})}(u) = P_{\A(T)}(u)$.
\end{proof}

Using Theorem~\ref{t:strong} and Lemma~\ref{l:twy15:demiclosed}, 
we obtain the following corollary; 
see Takahashi, Wong, and Yao~\cite{MR3397117}*{Theorem 3.2}.

\begin{corollary}\label{c:Takahashi}
 Let $H$, $C$, $\{\alpha_n \}$, and $\{\beta_n\}$ be the same as in
 Theorem~\ref{t:strong}. Let 
 $T\colon C \to C$ be a generalized hybrid mapping with an attractive
 point and $\{x_n\}$ a sequence in $C$ defined by $u, x_1 \in C$ and
 \eqref{e:halpern} for $n \in \N$. 
 Then $\{x_n\}$ converges strongly to $P_{\A(T)}(u)$. 
\end{corollary}

\begin{proof}
 Lemma~\ref{l:twy15:demiclosed} shows that $\Fhat(T) \subset \A(T)$. 
 Thus Theorem~\ref{t:strong} implies the conclusion. 
\end{proof}

\begin{remark}
 Corollary~\ref{c:Takahashi} is almost the same
 as~\cite{MR3397117}*{Theorem 3.2}, except that $\{\alpha_n\}$ and
 $\{\beta_n\}$ are assumed to be sequences in $(0,1)$ 
 in~\cite{MR3397117}*{Theorem 3.2}. 
\end{remark}

\begin{theorem}\label{t:weak}
 Let $H$ be a Hilbert space, $C$ a nonempty convex subset of $H$, 
 $T\colon C \to C$ a mapping with an attractive point, 
 $\{\alpha_n \}$ a sequence in $[0,1]$, 
 and $\{x_n\}$ a sequence in $C$ defined by $x_1 \in C$ and
 \begin{equation}\label{e:mann}
  x_{n+1} = \alpha_n x_n + (1-\alpha_n) Tx_n
 \end{equation}
 for $n \in \N$. 
 Suppose that $\liminf_n \alpha_n (1-\alpha_n) >0$. 
 If $\Fhat(T) \subset \A(T)$, 
 then $\{x_n\}$ converges weakly to some point in $\A(T)$. 
\end{theorem}

\begin{proof}
 Let $\tilde{T}\colon H\to H$ be an extension of $T$
 defined by~\eqref{e:extension}.
 As in the proof of Theorem~\ref{t:strong}, 
 Lemma~\ref{l:demiclosed} shows that a mapping $\tilde{T}$ is a
 quasinonexpansive extension of $T$, 
 and that $\Fhat(\tilde{T}) = \F(\tilde{T}) = \A(T)$. 
 We can also check that 
 \[
 x_{n+1} = \alpha_n x_n + (1-\alpha_n) \tilde{T}x_n
 \]
 for all $n \in \N$. 
 Therefore Theorem~\ref{t:MT2004} implies the conclusion. 
\end{proof}

Finally, we obtain a weak convergence result for a widely more
generalized hybrid mapping in the sense of \cite{MR3131129} as a
corollary of Theorem~\ref{t:weak}. 

Let $C$ be a nonempty subset of a Hilbert space $H$ and 
$T\colon C \to H$ a mapping. 
Recall that $T$ is \emph{widely more generalized hybrid} \cite{MR3131129}
if there exist 
$\alpha,\beta,\gamma,\delta,\epsilon,\zeta,\eta \in \R$ such that 
\begin{multline}\label{e:wmgh}
 \alpha \norm{Tx-Ty}^2 + \beta \norm{x-Ty}^2 + \gamma \norm{Tx-y}^2
 + \delta \norm{x-y}^2 \\
 +\epsilon \norm{x-Tx}^2 + \zeta \norm{y-Ty}^2 
 + \eta \norm{x-Tx - (y-Ty)}^2  \leq 0
\end{multline}
for all $x,y \in C$. 
Such a mapping $T$ is called an 
$(\alpha,\beta,\gamma,\delta,\epsilon,\zeta,\eta)$-widely more
generalized hybrid mapping. 

Using Theorem~\ref{t:weak} and \cite{MR3073499}*{Lemma 11}, we obtain
the following corollary; 
see~\cite{MR3073499}*{Theorem 14}. 

\begin{corollary}\label{c:Guu}
 Let $H$, $C$, and $\{\alpha_n \}$ be the same as in
 Theorem~\ref{t:weak}. 
 Let $T\colon C \to C$ be an
 $(\alpha,\beta,\gamma,\delta,\epsilon,\zeta,\eta)$-widely more
 generalized hybrid mapping with an attractive point and $\{x_n\}$ a
 sequence in $C$ defined by $x_1 \in C$ and~\eqref{e:mann}
 for $n \in \N$. Suppose that 
 \begin{equation}\label{e:wmgh2qn}
 \alpha + \beta + \gamma + \delta \geq 0,\, 
  \alpha + \gamma > 0, \text{ and } \epsilon + \eta \geq 0
 \end{equation}
 hold. 
 Then $\{x_n\}$ converges weakly to some point in $\A(T)$. 
\end{corollary}

\begin{proof}
 \cite{MR3073499}*{Lemma 11} shows that $\Fhat(T) \subset \A(T)$. 
 Thus Theorem~\ref{t:weak} implies the conclusion. 
\end{proof}

\begin{remark}
 Corollary~\ref{c:Guu} is almost the same 
 as~\cite{MR3073499}*{Theorem 14}, except that
 $\alpha,\beta,\gamma,\delta,\epsilon,\zeta$, and $\eta$ are assumed to
 satisfy \eqref{e:wmgh2qn} or 
 \begin{equation}\label{e:wmgh2qn2}
  \alpha + \beta + \gamma + \delta \geq 0,\, 
   \alpha + \beta > 0, \text{ and } \zeta + \eta \geq 0
 \end{equation}
 in \cite{MR3073499}*{Theorem 14}. 
 We can confirm that the conditions~\eqref{e:wmgh2qn} 
 and~\eqref{e:wmgh2qn2} are equivalent for an
 $(\alpha,\beta,\gamma,\delta,\epsilon,\zeta,\eta)$-widely more
 generalized hybrid mapping. 
\end{remark}

\section*{Acknowledgment}

The author would like to acknowledge the financial support from
Professor Kaoru Shimizu of Chiba University.


\begin{bibdiv}
\begin{biblist}

\bib{MR2983907}{article}{
      author={Akashi, Shigeo},
      author={Takahashi, Wataru},
       title={Strong convergence theorem for nonexpansive mappings on
  star-shaped sets in {H}ilbert spaces},
        date={2012},
        ISSN={0096-3003},
     journal={Appl. Math. Comput.},
      volume={219},
       pages={2035\ndash 2040},
         url={https://doi.org/10.1016/j.amc.2012.08.046},
}

\bib{MR2682871}{article}{
      author={Aoyama, Koji},
      author={Iemoto, Shigeru},
      author={Kohsaka, Fumiaki},
      author={Takahashi, Wataru},
       title={Fixed point and ergodic theorems for {$\lambda$}-hybrid mappings
  in {H}ilbert spaces},
        date={2010},
        ISSN={1345-4773},
     journal={J. Nonlinear Convex Anal.},
      volume={11},
       pages={335\ndash 343},
}

\bib{MR2960628}{article}{
      author={Aoyama, Koji},
      author={Kimura, Yasunori},
      author={Kohsaka, Fumiaki},
       title={Strong convergence theorems for strongly relatively nonexpansive
  sequences and applications},
        date={2012},
        ISSN={1906-9685},
     journal={J. Nonlinear Anal. Optim.},
      volume={3},
       pages={67\ndash 77},
}

\bib{MR2981792}{article}{
      author={Aoyama, Koji},
      author={Kohsaka, Fumiaki},
       title={Fixed point and mean convergence theorems for a family of
  {$\lambda$}-hybrid mappings},
        date={2011},
        ISSN={1906-9685},
     journal={J. Nonlinear Anal. Optim.},
      volume={2},
       pages={87\ndash 95},
}

\bib{MR3017202}{article}{
      author={Aoyama, Koji},
      author={Kohsaka, Fumiaki},
       title={Uniform mean convergence theorems for hybrid mappings in
  {H}ilbert spaces},
        date={2012},
        ISSN={1687-1812},
     journal={Fixed Point Theory Appl.},
       pages={2012:193, 13},
         url={http://dx.doi.org/10.1186/1687-1812-2012-193},
}

\bib{sqnIII}{article}{
      author={Aoyama, Koji},
      author={Kohsaka, Fumiaki},
       title={Strongly quasinonexpansive mappings, {III}},
        date={2020},
     journal={Linear Nonlinear Anal.},
      volume={6},
       pages={1\ndash 12},
}

\bib{MR2529497}{article}{
      author={Aoyama, Koji},
      author={Kohsaka, Fumiaki},
      author={Takahashi, Wataru},
       title={Strongly relatively nonexpansive sequences in {B}anach spaces and
  applications},
        date={2009},
        ISSN={1661-7738},
     journal={J. Fixed Point Theory Appl.},
      volume={5},
       pages={201\ndash 224},
         url={http://dx.doi.org/10.1007/s11784-009-0108-7},
}

\bib{MR0298499}{article}{
      author={Dotson, W.~G., Jr.},
       title={Fixed points of quasi-nonexpansive mappings},
        date={1972},
        ISSN={0263-6115},
     journal={J. Austral. Math. Soc.},
      volume={13},
       pages={167\ndash 170},
}

\bib{MR1074005}{book}{
      author={Goebel, Kazimierz},
      author={Kirk, W.~A.},
       title={Topics in metric fixed point theory},
      series={Cambridge Studies in Advanced Mathematics},
   publisher={Cambridge University Press, Cambridge},
        date={1990},
      volume={28},
        ISBN={0-521-38289-0},
         url={http://dx.doi.org/10.1017/CBO9780511526152},
}

\bib{MR3073499}{article}{
      author={Guu, Sy-Ming},
      author={Takahashi, Wataru},
       title={Existence and approximation of attractive points of the widely
  more generalized hybrid mappings in {H}ilbert spaces},
        date={2013},
        ISSN={1085-3375},
     journal={Abstr. Appl. Anal.},
       pages={Art. ID 904164, 10},
         url={https://doi.org/10.1155/2013/904164},
}

\bib{MR3131129}{article}{
      author={Kawasaki, Toshiharu},
      author={Takahashi, Wataru},
       title={Existence and mean approximation of fixed points of generalized
  hybrid mappings in {H}ilbert spaces},
        date={2013},
        ISSN={1345-4773},
     journal={J. Nonlinear Convex Anal.},
      volume={14},
       pages={71\ndash 87},
}

\bib{MR2761610}{article}{
      author={Kocourek, Pavel},
      author={Takahashi, Wataru},
      author={Yao, Jen-Chih},
       title={Fixed point theorems and weak convergence theorems for
  generalized hybrid mappings in {H}ilbert spaces},
        date={2010},
        ISSN={1027-5487},
     journal={Taiwanese J. Math.},
      volume={14},
       pages={2497\ndash 2511},
         url={https://doi.org/10.11650/twjm/1500406086},
}

\bib{MR2058234}{article}{
      author={Matsushita, Shin-ya},
      author={Takahashi, Wataru},
       title={Weak and strong convergence theorems for relatively nonexpansive
  mappings in {B}anach spaces},
        date={2004},
        ISSN={1687-1820},
     journal={Fixed Point Theory Appl.},
       pages={37\ndash 47},
         url={http://dx.doi.org/10.1155/S1687182004310089},
}

\bib{MR2773244}{article}{
      author={Nilsrakoo, Weerayuth},
      author={Saejung, Satit},
       title={Strong convergence theorems by {H}alpern-{M}ann iterations for
  relatively nonexpansive mappings in {B}anach spaces},
        date={2011},
        ISSN={0096-3003},
     journal={Appl. Math. Comput.},
      volume={217},
       pages={6577\ndash 6586},
         url={https://doi.org/10.1016/j.amc.2011.01.040},
}

\bib{MR1386686}{incollection}{
      author={Reich, Simeon},
       title={A weak convergence theorem for the alternating method with
  {B}regman distances},
        date={1996},
   booktitle={Theory and applications of nonlinear operators of accretive and
  monotone type},
      series={Lecture Notes in Pure and Appl. Math.},
      volume={178},
   publisher={Dekker, New York},
       pages={313\ndash 318},
}

\bib{MR2548424}{book}{
      author={Takahashi, Wataru},
       title={Introduction to nonlinear and convex analysis},
   publisher={Yokohama Publishers, Yokohama},
        date={2009},
        ISBN={978-4-946552-35-9},
}

\bib{MR2858695}{article}{
      author={Takahashi, Wataru},
      author={Takeuchi, Yukio},
       title={Nonlinear ergodic theorem without convexity for generalized
  hybrid mappings in a {H}ilbert space},
        date={2011},
        ISSN={1345-4773},
     journal={J. Nonlinear Convex Anal.},
      volume={12},
       pages={399\ndash 406},
}

\bib{MR3015118}{article}{
      author={Takahashi, Wataru},
      author={Wong, Ngai-Ching},
      author={Yao, Jen-Chih},
       title={Attractive point and weak convergence theorems for new
  generalized hybrid mappings in {H}ilbert spaces},
        date={2012},
        ISSN={1345-4773},
     journal={J. Nonlinear Convex Anal.},
      volume={13},
       pages={745\ndash 757},
}

\bib{MR3397117}{article}{
      author={Takahashi, Wataru},
      author={Wong, Ngai-Ching},
      author={Yao, Jen-Chih},
       title={Attractive points and {H}alpern-type strong convergence theorems
  in {H}ilbert spaces},
        date={2015},
        ISSN={1661-7738},
     journal={J. Fixed Point Theory Appl.},
      volume={17},
       pages={301\ndash 311},
         url={https://doi.org/10.1007/s11784-013-0142-3},
}

\end{biblist}
\end{bibdiv}
\end{document}